\documentclass[graybox]{svmult}

\usepackage{type1cm}        
\usepackage{makeidx}         
\usepackage{graphicx}        
\usepackage{multicol}        
\usepackage[bottom]{footmisc}

\usepackage{newtxtext}       %
\usepackage[varvw]{newtxmath}       

\makeindex             

\renewcommand{\div}{{\hbox{div}}}
\newcommand{\dx}{\, dx}
\newcommand{\ds}{\, ds}
\newcommand{\T}{\mathcal{T}_h}
\newcommand{\C}{\mathcal{C}_h}
\newcommand{\IRT}{\mathcal{IRT}}
\newcommand{\DD}{\mathcal{D}_h}
\newcommand{\M}{\mathcal{M}_h}
\newcommand{\F}{\mathcal{F}_h}
\newcommand{\N}{\mathcal{N}_h}

\newcommand{\nrmf}[1]{\|#1\|_{\mathcal{M}_h}}
\newcommand{\nrmD}[1]{\|#1\|_{\mathcal{D}_h}}

\begin{document}

\title*{Conservative flux reconstruction for an elliptic interface  problem using CutFEM}
\author{Aimene Gouasmi\orcidID{0009-0001-1897-9160} and\\ Daniela Capatina\orcidID{0009-0001-0380-0412}}
\institute{Aimene Gouasmi \at LMAP \& CNRS UMR 5142, University of Pau and Pays de l'Adour, IPRA BP 1155, 64013 Pau, France.  \email{aimene.gouasmi@univ-pau.fr}
	\and Daniela Capatina \at LMAP \& CNRS UMR 5142, University of Pau and Pays de l'Adour, IPRA BP 1155, 64013 Pau, France.  \email{daniela.capatina@univ-pau.fr}}
%
%
\maketitle

\abstract*{This paper deals with the local recovery of conservative fluxes for an elliptic interface problem with discontinuous coefficients. The transmission conditions on the interface are imposed weakly and the discretisation is achieved by using conforming finite elements on unfitted meshes, with the aid of the CutFEM method. In a first attempt at flux reconstruction, we define a flux belonging to the Raviart-Thomas space of each sub-domain following the method developed for a boundary problem. However, the transmission condition is not satisfied by the recovered flux. In order to overcome this shortcoming, we propose a second approach, where the flux belongs to the recently introduced immersed Raviart-Thomas space. This ensures both the continuity of the normal flux across the interface and a natural conservation property on the cut cells. Subsequently, we apply the recovered flux to a posteriori error analysis and prove the sharp reliability of the error estimator.
}
\abstract{This paper deals with the local recovery of conservative fluxes for an elliptic interface problem with discontinuous coefficients. The transmission conditions on the interface are imposed weakly and the discretisation is achieved by using conforming finite elements on unfitted meshes, with the aid of the CutFEM method. In a first attempt at flux reconstruction, we define a flux belonging to the Raviart-Thomas space of each sub-domain following the method developed for a boundary problem. However, the transmission condition is not satisfied by the recovered flux. In order to overcome this shortcoming, we propose a second approach, where the flux belongs to the recently introduced immersed Raviart-Thomas space. This ensures both the continuity of the normal flux across the interface and a natural conservation property on the cut cells. Subsequently, we apply the recovered flux to a posteriori error analysis and prove the sharp reliability of the error estimator.
}
\section{Introduction}

The importance of reconstructing conservative local fluxes from primal discrete solutions is well acknowledged in the literature, see for instance \cite{Ern_flux,Dana2016,Aimene,Capatina-He}. Such  fluxes play an important role in  a posteriori error analysis \cite{Aposteriori1,Aposteriori2}, where the difference between  the  numerical flux and a recovered equilibrated flux provides a reliable error indicator, which can be further used in adaptive mesh refinement procedures.

This paper focuses on a $2$D second-order elliptic interface problem, characterized by discontinuous coefficients and \textcolor{black}{standard transmission conditions for the solution and the normal flux} on the interface. 

The Cut Finite Element Method \cite{CutFEM2}, designed for  the case where the mesh does not align with the interface, is employed  for numerical approximation. Our goal is to recover local conservative fluxes in the Raviart-Thomas space, extending a methodology introduced in \cite{Dana2016} for the Poisson equation on fitted meshes. The construction is based on a mixed problem, with the primal solution coinciding with the original finite element solution and with the multiplier, defined on the \textcolor{black}{edges} of the mesh, being naturally used to define the degrees of freedom of the fluxes. It is important  to note that the multiplier can be computed locally, by solving an explicit low-dimensional linear system for each vertex. This approach has recently been extended to a boundary problem on \textcolor{black}{an} unfitted mesh in \cite{Capatina-He}.
 
Here, we treat the interface diffusion problem discretized by \textcolor{black}{piecewise linear} conforming elements on \textcolor{black}{an} unfitted mesh. We  first show how to reconstruct a $H(div)-$flux on each sub-domain $\Omega^i$ \textcolor{black}{($1\le i\le 2$)}, based on a hybrid mixed formulation with Lagrange multipliers associated to each $\Omega^i$. On the cut elements, we  define the multipliers, and hence the normal fluxe\textcolor{black}{s}, on the whole edges. In each $\Omega^i$, the local conservation on a cut element is achieved with respect to an extension of the data $f^i$ to the whole element. We thus obtain a robust reconstruction with respect to both the interface/mes\textcolor{black}{h} geometry  and the diffusion coefficients. However, this method  does not respect the  transmission condition on the normal flux. To overcome this inconvenience, we propose an alternative reconstruction, based on the  immersed Raviart-Thomas space recently introduced in \cite{IRT}. This space  insures that  the transmission condition is strongly satisfied, and in addition, the conservation property is naturally satisfied  on the cut cells. Finally,  we use this recovered flux in the a posteriori error analysis and prove the sharp reliability of the error estimator.

\section{The continuous  and  discrete problems}
\label{sec:1}

Let $\Omega$ be a 2D polygonal domain and $\Gamma$  an interface separating $\Omega$ into two disjoint sub-domains: $\bar{\Omega}=\bar{\Omega}^{1}\cup \bar{\Omega}^{2}$, $\partial\Omega^{1}\cap \partial\Omega^{2}=\Gamma$. We denote by $n_{\Gamma}$ the unit normal vector to $\Gamma$ oriented from $\Omega^1$ to $\Omega^2$.
We consider the following  model problem:
\begin{equation}\label{eq: continuous_problem}
	\left\{
	\begin{split} 
		-div(K\nabla u^{i})=f^i \quad &\text{in }\Omega^i,\,\, i=1,2 \\
		u=0 \quad & \text{on }\partial \Omega \\ 
		[u]=0,\,\,[K\nabla u\cdot n_{\Gamma}]=\textcolor{black}{0}  \quad &\text{on }\Gamma,
	\end{split}
	\right.
\end{equation}
where $[u]=u^1-u^2$ is the jump across $\Gamma$. We suppose that $f^i\in L^2(\Omega^i)$ and for the sake of simplicity, we assume here that  $K|_{\Omega^i}=k_i>0$, for $i=1,2$. The approach can be extended to piecewise constant positive definite tensors $K$, \textcolor{black}{as well as to a non-zero jump of the normal fluxes on the interface.}

In view of the finite element approximation of \eqref{eq: continuous_problem}, we introduce some notation. Let $\mathcal{T}_h$ be a triangular regular mesh of $\Omega$\textcolor{black}{, whose elements are closed sets}. We denote by $\mathcal{F}_h$ the set of \textcolor{black}{edges}. The diameter of $T\in \T$ (and \textcolor{black}{the length} of $F\in \F$) is denoted by $h_T$ (and $h_F$). For an interior \textcolor{black}{edge} $F$, $n_F$ denotes a fixed, unit normal  vector to $F$, oriented from $T_F^-$ towards $T_F^+$, with $T_F^-$, $T_F^+$ the two triangles sharing the \textcolor{black}{edge} $F$. If $F\subset \partial \Omega$, we take $n_F$ the outward normal vector to $\Omega$, whereas for $F\subset \Gamma$, we set $n_F=n_{\Gamma}$. For $\omega\subset \mathbb R^d$ with $1\leq d \leq 2$, we denote by $\|\cdot\|_\omega$ the $L^2(\omega)$-norm and by $\pi^{\textcolor{black}{m}}_{\omega}$ the $L^2(\omega)$-orthogonal projection on $P^{\textcolor{black}{m}}(\omega)$, for $\textcolor{black}{m}\in \mathbb{N}$.
We consider, for $i=1,2$, 
\begin{equation*}
\mathcal{T}_h^{i} = \big\{ T \in \mathcal{T}_h \ ; \ T\cap \Omega^{i}\neq \emptyset \big\},\quad \mathcal{F}_h^{i} = \big\{ F \in \mathcal{F}_h \ ; \ F\cap \Omega^{i}\neq \emptyset \big\}
\end{equation*}
and define       
$\Omega_h^i=\bigcup_{T\in \mathcal{T}_h^i}T$. Note that $\Omega^i\subset \Omega_h^i$. Regarding the cut elements, we put 
\begin{equation*}
	\begin{split}
		\mathcal{T}_h^{\Gamma}= \big\{ T\in\mathcal{T}_h \ ;\ T\cap \Gamma\neq \emptyset \big\},&\quad   \mathcal{F}_h^{\Gamma} =\big\{ F \in \F \ ;\ F\cap \Gamma \neq \emptyset\ \big\},\\
		\mathcal{F}_g^i =\big\{ F\in\mathcal{F}_h^i \ ;\ (T_F^+\cup T_F^-)\cap\Gamma \neq \emptyset \big\},&\quad T^i=T\cap\Omega^i \,\,\, \forall T\in \T^{\Gamma}.
	\end{split}
\end{equation*}

\textcolor{black}{In order to focus on the flux reconstruction, we assume in this paper that the interface $\Gamma$ is a polygonal line, more specifically that } for each $T\in \T^{\Gamma}$, the intersection $\Gamma_T=T \cap \Gamma $ is a line. For a function $v$ discontinuous across $\Gamma$, we denote $v^i=v_{|\Omega^i}$ and we define the following two means at $x\in \Gamma$:
$$ \{v\}(x)=\omega_{1} v^{1}(x)+\omega_{2}v^{2}(x),\quad \{v\}^*(x)=\omega_{2} v^{1}(x)+\omega_{1}v^{2}(x),$$ 
where the weights $\omega_{1}$, $\omega_{2}$ as well as the harmonic mean $k_{\Gamma}$ are given (cf. \cite{Ern}) by: 
$$ \omega_{1}=\frac{k_{2}}{k_{2}+k_{1}},\quad \quad \omega_{2}=\frac{k_{1}}{k_{1}+k_{2}}, \quad k_{\Gamma}=\frac{k_{1}k_{2}}{k_{1}+k_{2}}. $$
It is useful to introduce the arithmetic mean $\langle v\rangle=\frac{1}{2}(v^{-}+v^+)$ and the jump $[\![v]\!] =v^{-}-v^+$ across an interior \textcolor{black}{edge} $ F\in \F^i$; for a boundary \textcolor{black}{edge}, we set $\langle v\rangle=[\![v]\!]=v$.   Finally, for $i=1,2$, let
$V^i=\big\{ v \in H^1(\Omega_h^i);\, v_{|(\partial\Omega^i\setminus\Gamma)}=0\big\}$
and
$$\C^i=\big\{ v \in V^i:\,  v_{|T}\in P^1(T),\,\, \forall \ T\in  \mathcal{T}_h^{i} \big\},\quad \C=\C^{1}\times \C^{2}. $$

\textcolor{black}{We consider the following weak formulation of problem \eqref{eq: continuous_problem}: Find  $u \in H^1_0(\Omega) $ s.t.
\begin{equation}\label{eq: continuous_problem_weak}
\int_{\Omega}K\nabla u \cdot \nabla v\dx= \int_{\Omega}f v\dx\quad \forall v\in H^1_0(\Omega).
\end{equation}
}Regarding \textcolor{black}{its} numerical approximation, we use Nit\textcolor{black}{s}che's method to take into account the transmission conditions on $\Gamma$ \cite{Nitche}. Moreover, we use CutFEM \cite{CutFEM2} to stabilize the approach with respect to the geometry of the interface, by adding a ghost penalty term. The discrete problem reads: Find  $u_h=(u_h^1,u_h^2) \in \C $ s.t.
\begin{equation}\label{eq:Primal_conform_Formulation}
	a_h(u_h,v_h)=l_h(v_h) \quad \forall v_h\in \C,
\end{equation}
where  $ a_h(u_h,v_h)=\displaystyle{\sum_{i=1}^{2}}\bigg(a_i(u_h^i,v_h^i)+\beta j_i(u_h^i,v_h^i)\bigg)+a_{\Gamma}(u_h,v_h)$ with:
\begin{equation*}
\begin{split}
a_i(u_h^i,v_h^i)=&\textcolor{black}{\sum_{T\in\T^{i}}}\int_{\textcolor{black}{T\cap}\Omega^i}k_i\nabla u_h^i \cdot \nabla v_h^i\dx,\quad  j_i(u_h^i,v_h^i)=\sum_{F\in\mathcal{F}_g^i} h_F \int_{F}k_i[\![\partial_n u_h^{i}]\!][\![\partial_n v_h^{i}]\!]\ds,\\
a_{\Gamma}(u_h,v_h)=& \sum_{T\in\T^{\Gamma}}\int_{\Gamma_T}( \frac{\gamma k_{\Gamma}}{h_T}[u_h][v_h]-\{K\nabla u_h\cdot n_{\Gamma}\}[v_h]-\{K\nabla v_h\cdot n_{\Gamma}\}[u_h])\ds,\\
l_h(v_h)=& \sum_{i=1}^{2}\int_{\Omega^i}f^i v_h^i\dx. 
\end{split}
\end{equation*}
The stabilization parameters $\gamma,\beta\textcolor{black}{>0}$ can be chosen independently of the mesh/interface geometry and of the diffusion coefficients. For any $v_h\in \C$, we define the norm:
\begin{equation*}
	\|v_h\|_h^2= \sum_{i=1}^{2}\bigg(k_i\|\nabla v_h^i\|^2_{\Omega^i}+j_i(v_h^i,v_h^i)\bigg)
	+\sum_{T\in \mathcal{T}_h^{\Gamma}}\int_{\Gamma_T}\frac{k_{\Gamma}}{h_T}[v_h]^2 \ds.
\end{equation*}

For $\gamma$ large enough, it is well known that $a_h(\cdot,\cdot)$ is uniformly $(\C,\|\cdot\|_h)$-coercive. Then problem \eqref{eq:Primal_conform_Formulation} is well-posed thanks to the Lax-\textcolor{black}{Milgram} theorem. 

\section{The auxiliary mixed formulation}
\label{sec:2}
Let $\DD=\DD^1\times \DD^2$ and $\M=\M^1\times \M^2$, where for $i=1,2$ we set:
\begin{equation*}
	\begin{split}
		\DD^i=&\{v\in L^2(\T^i);\, {v}_{|T}\in P^1(T)\,\forall T\in \T^i\},\\
		\M^i=&\{\mu\in L^2(\F^i);\, {\mu}|_{F}\in P^1(F)\, \forall F\in \F ^i,\, 
		\displaystyle\sum_{F\in\mathcal{F}_N}\mathfrak{s}_N^Fh_F{\mu}_{|F}(N)=0 \,\forall N\in \overset{\circ}{\mathcal {N}_h^i}\}.
	\end{split}     
\end{equation*}
Here above, $\overset{\circ}{\mathcal {N}_h^i}$ denotes the set of nodes interior to $\Omega_h^i$, $\mathcal{F}_N$ the set of \textcolor{black}{edges} sharing the node $N$, and $\mathfrak{s}_N^F\textcolor{black}{=sign(n_F, N)}$ \textcolor{black}{is equal to $1$ or $-1$}, depending upon the orientation of $n_F$ with respect to the clockwise rotation sense around $N$. These spaces are endowed with the following norms:
\begin{equation*}
	\begin{split}
		\nrmD{v_h}^2&=\|v_h\|_h^2+ \sum_{i=1}^{2}\sum_{F\in \F^i}\int_F k_i h_F^{-1}[\![v_h^i]\!]^2\ds,\quad \forall v_h\in \DD,\\
		\nrmf{\mu_h}^2 &=\sum_{i=1}^{2}\sum_{F\in \F^i}\int_{F}k_i h_F(\mu_h^i)^2\ds,\quad \forall \mu_h \in \M.      
	\end{split}
\end{equation*}
Following \cite{Dana2016,Aimene}, we consider  the mixed formulation: Find $(\tilde{u}_h,\theta_h)\in \DD\times \M$,
\begin{equation}\label{eq: Mixed formulation}
	\begin{split}
		\tilde {a}_h(\tilde{u}_h,v_h)+&b_h(\theta_h,v_h)=l_h(v_h)\quad \forall v_h \in \mathcal{D}_h,\\
		&b_h(\mu_h,\tilde{u}_h)=0 \quad \qquad\forall \mu_h \in \mathcal{M}_h,
	\end{split}
\end{equation}
where $\tilde{a}_h(\cdot,\cdot)= a_h(\cdot,\cdot)- d_h(\cdot,\cdot)$ and $ b_h(\mu_h,v_h)=\displaystyle{\sum_{i=1}^{2}} b_h^i(\mu_h^i,v_h^i)$, with:

\begin{equation*}
	\begin{split}
		d_h(\tilde u_h,v_h)=&\sum_{i=1}^{2} \sum_{F\in\mathcal{F}_h^{i}}  \int_{F\cap\Omega^i} \bigg(\langle k_i\nabla \tilde {u}_h^i\cdot n_F\rangle [\![v_h^i]\!]   +\langle k_i\nabla v_h^i \cdot n_F\rangle [\![\tilde{u}_h^i]\!]\bigg)\ds,\\
		b_h^i(\mu_h^i,v_h^i)=&\sum_{F\in\mathcal{F}_h^{i}} \frac{k_i h_F}{2}\sum_{N\in\mathcal{N}_F}{\mu_h^i}_{|F}(N)[\![v_h^i]\!](N) \approx \sum_{F\in\mathcal{F}_h^{i}}  \int_{F}k_i\mu_h^i[\![v_h^i]\!]\ds.
	\end{split}
\end{equation*}
Here above, $\mathcal{N}_F$ denotes the set of nodes belonging to $F$. We obtain similar results to \cite{Dana2016,Aimene}, which are presented below. We refer the readers to \cite{Dana2016,Aimene} for more details.    

\begin{theorem}\label{thrm: inf-sup_cond}
	There exists a constant $\tilde{\beta}$ independent of $h$ and $K$ such that 
	\begin{equation*}
		\inf_{\mu_h\in \mathcal{M}_h}\sup_{v_h\in \mathcal{D}_h} \frac{b_h(\mu_h,v_h)}{\nrmf{\mu_h}\nrmD{v_h}}\geq \tilde{\beta}.
	\end{equation*}
\end{theorem}
\begin{theorem}\label{thrm: kernel_b_h }
	The discrete kernel of $b_h(\cdot,\cdot)$ coincides with the space $\C$, i.e, 
	\begin{equation*}
		\mathrm{Ker}\,b_h=\left\{ v_h\in\mathcal{D}_h;\, b_h(\mu_h,v_h)=0,\, \forall \mu_h \in \mathcal{M}_h \right\}=\mathcal{C}_h.
	\end{equation*}    
\end{theorem}
Theorem \ref{thrm: kernel_b_h } gives the coercivity of $\tilde{a}_h(\cdot,\cdot)$ on $\mathrm{Ker}\,b_h$, as well as the equivalence between the primal and mixed formulations \eqref{eq:Primal_conform_Formulation} and \eqref{eq: Mixed formulation}. Indeed, one gets that $u_h=\tilde{u}_h$. The Babuska-Brezzi theorem next yields the well-posedness of \eqref{eq: Mixed formulation}. Note that the multipliers $\theta_h^i$ and the forms $b_h^i(\cdot,\cdot)$ are defined on the whole \textcolor{black}{edges; thus, on a cut edge $F\in \mathcal{F}_h^{\Gamma} $ we have two multipliers, both defined on $F$ and not on $F^i=F\cap \Omega^i$.}

A crucial feature of the method is that each multiplier  $\theta_h^i$  can be computed locally, as sum of local contributions defined on patches associated  to the nodes. More precisely, for $i\in \{1,2\} $ we showed that $\theta_h^i=\sum_{N\in \N^i} \theta_N^i$,  where $\theta_N^i\in \M^i$ lives on $\mathcal F_N\cap \mathcal F_h^i$ and is the unique solution of the following low-order linear system:
\begin{equation*}
	\forall T\in \omega_N : \left\{
	\begin{split}
		b_h^i(\theta_N^i,\varphi_N\chi_T)&=r_h^i(\varphi_N\chi_T),\label{theta1}\\
		b_h^i(\theta_N^i,\varphi_M\chi_T)&=0\quad \forall M\in \mathcal{N}_T\setminus \{N\}
	\end{split}\right.
\end{equation*}
where \textcolor{black}{$\chi_T$ is the characteristic function of $T$,} $r_h^1(v_h^1)=r_h((v_h^1,0))$, $r_h^2(v_h^2)=r_h((0,v_h^2))$ with $r_h(\cdot)=l_h(\cdot)-\tilde a_h(u_h,\cdot)$, and where $\omega_N$  is the set of triangles sharing the node $N$. For more details, see \cite{Aimene}. 

\section{Local flux reconstruction}

In the sequel, we propose two reconstructions of conservative fluxes (approximations of $\sigma=K\nabla u$). One follows the approach introduced in \cite{Capatina-He} for an unfitted boundary problem, while the other one uses the immersed Raviart-Thomas space  of \cite{IRT}.
In order to simplify the presentation, we assume here that no \textcolor{black}{edge} is situated on  $\Gamma$.

\subsection{Flux recovering in $\mathcal{RT}^m(\Omega_h^1)\times \mathcal{RT}^m(\Omega_h^2) $}
We begin by reconstructing a flux on each sub-domain $\Omega_h^i$. For $i\in \{1,2\}$, we define $\sigma_h^i\in \mathcal{RT}^m(\Omega_h^i)$ with $m=0$ or $1$, by imposing its degrees of freedom as follows. 

For $F\in \mathcal{F}_h^i$, we set as in \cite{Dana2016, Aimene}
\[ 
\int_{F}\sigma_h^i\cdot n_F w\ds=\int_{F} \langle k_i\nabla u_h^i\cdot n_F\rangle w\ds - \frac{k_i h_F}{2}\sum_{N\in\mathcal{N}_F}{(\theta_h^i}_{|F}w)(N),\quad \forall w\in P^m(F),
\]
whereas for $F\subset \partial \Omega_h^i $ not situated in $\Omega^i$, we simply set $\sigma_h^i\cdot n_F= k_i\nabla u_h^i\cdot n_F$.

 Finally, if $m=1$ we also add interior degrees of freedom: for any $T\in \mathcal T_h^i$,
\begin{equation*}
	\begin{split}
		\int_{T}\sigma_h^i\cdot \zeta \dx=&\int_{T}k_i\nabla u_h^i\cdot \zeta\dx - \int_{\Gamma_T} \omega_ik_i \zeta\cdot n_{\Gamma} [u_h]\ds \\
		&+ \sum_{F\in \mathcal F_g^i\cap \partial T}\beta h_F\int_{F}k_i[\![\partial_n u_h^i]\!] [\![\zeta\cdot n_F]\!]\ds,\quad \forall \zeta\in (P^{0}(T))^2.
	\end{split}
\end{equation*}

In order to establish the conservation property on the cut cells $T\in \T^{\Gamma}$, we need to extend $f^{i}$ (which is only defined on $T^i$ ) to the whole triangle $T$. Let  $T^i_C=T\setminus T^i$ and $F^i_C=F\setminus F^i$, where $F^i=F\cap\Omega^i$.  Following \cite{Capatina-He}, we define $f^{i}|_{T^i_C}$ in $P^m(T^i_C)$ by:  
\begin{equation*}
	\begin{split}
		\int_{T^i_C}f^ip\dx=&\int_{\Gamma_T}\bigg(\textcolor{black}{[k\nabla u_h\cdot n_{\Gamma}](\omega_i-1)}+\frac{(-1)^i\gamma k_{\Gamma}}{h_T}[u_h]\bigg)p\ds\\
		&+\frac{1}{2}\int_{F^i_C }[\![k_i\partial_n u_h^i]\!]p\ds, \quad \forall p\in P^m(T^i_C).
	\end{split}
\end{equation*}
The proof of the conservation property on the non-cut cells is standard; on a cut cell $T\in \mathcal T_h^{\Gamma}$, we employ as usually Green's formula for  $\displaystyle{\int_{T}}(\div \sigma_h^i)v\dx$ with $v\in P^m(T)$, and we further test the mixed formulation \textcolor{black}{(\ref{eq: Mixed formulation})} with $(v\chi_T,0)$ if $i=1$, and with $(0, v\chi_T)$ if $i=2$. Using that $\bar{T}=\bar{T}^i \cup \bar{T}^i_C$ and integrating by parts the term $\displaystyle{\int_{T^i_C}}k_i\nabla u_h^i\cdot \nabla v\dx$, we finally deduce the conservation property with respect to the extension of $f^i$:
\begin{equation*}
	-(\div\sigma_h^i)_{|T}=\pi^m_Tf^{i}, \quad \forall \ T\in \mathcal{T}_h^i,\,\, i\in\{1,2\}.
\end{equation*}
Note that the global flux $\tilde{\sigma}_h$, defined by $(\tilde{\sigma}_h)_{|\Omega^i}=\sigma_h^i$, does not satisfy the transmission condition on the interface:  $[\tilde{\sigma}_h \cdot n_{\Gamma}] \neq\textcolor{black}{0}$, that is $\tilde{\sigma}_h$ does not belong to $H(\div, \Omega)$. In order to overcome this inconvenience, we propose next a different reconstruction, based on the  Immersed Raviart-Thomas space recently introduced in \cite{IRT}. 

Only the lowest-degree space was considered in \cite{IRT}, so in what follows we restrict ourselves to the case $m=0$.
  
\subsection{Flux recovering in $\IRT^0(\Omega)$}

Let us  recall the definition of the Immersed Raviart-Thomas space of lowest-degree \cite{IRT}. For a cut cell $T\in \T^{\Gamma}$, the local space $\IRT^0(T)$ is the set of functions $\psi$ such that $\psi_{|T^i}=\psi^i$, with $\psi^i\in \mathcal{RT}^0(T)$ satisfying:
\begin{equation*}
	[\psi\cdot n_\Gamma]=0,\quad [k^{-1}\psi\cdot t_\Gamma]=0,\quad \div \psi^1=\div \psi^2.
\end{equation*}
\textcolor{black}{Their} local degrees of freedom are $\dfrac{1}{h_{F_j}}\displaystyle{\int_{F_j}}\psi\cdot n_{F_j} \ds$ for $j\in\{1,2,3\}$, where $F_j \subset \partial T$. 

The global immersed space $\IRT^0(\Omega)$ is then defined as the set of functions $\psi$ satisfying: $\psi_{|T}\in \mathcal{RT}^0(T)$ for $T\in \T\backslash\T^{\Gamma}$,  $\psi_{|T}\in \mathcal{IRT}^0(T)$ for $T\in \T^{\Gamma}$ and
\begin{equation}\label{eq:dof_IRT}
	\int_{F}[\![\psi\cdot n_F]\!]\ds=0,\quad \forall F\in \F^{int}.
\end{equation}
Note that a function $\psi$ of $\IRT^0(\Omega)$ satisfies strongly the transmission condition $ [\psi\cdot n_\Gamma]=0 $ across the interface $\Gamma$. However, the normal traces across the cut \textcolor{black}{edges} are only weakly continuous according to \eqref{eq:dof_IRT}.

We next  reconstruct a flux $\sigma_h\in \IRT^0(\Omega)$. The definition of the degrees of freedom on the non-cut cells is the same as previously; in particular, for $F\in \F^i\backslash \F^\Gamma$ we get 
\begin{equation}\label{eq: def_flux_IRT0_noncut}
\sigma_h\cdot n_F = \langle k_i\nabla u_h^i\cdot n_F\rangle- k_i\pi_F^0\theta_h^i.
\end{equation}
On a cut \textcolor{black}{edge} $F\in \F^{\Gamma}$, we now set:
\begin{equation}\label{eq: def_flux_IRT0}
	\int_{F}\sigma_h\cdot n_F\ds = \sum_{i=1}^{2}\left(\int_{F^i}\langle k_i\nabla u_h^i\cdot n_F\rangle\ds -\int_{F} k_i\theta_h^i\ds\right).
\end{equation}
Note that $\sigma_h\cdot n_F$ is only piecewise constant on the cut \textcolor{black}{edges}, but $\sigma_h$ belongs to $H(\div, T)$ for any $T\in \T$.
We can then establish the following conservation property.

\begin{theorem}\label{thrm: conservation_prop}
	Let $f\in L^2(\Omega)$ defined by $f_{|\Omega^i}=f^i$. Then one has that
	\begin{equation}\label{eq: conservation_property}
		-(\div \sigma_h)_{|T}=\pi^0_T f,\quad \forall T\in\T.
	\end{equation}
\end{theorem}
\begin{proof}
\textcolor{black}{Let $T\in\T$.} We start from  $\int_{T} \div \sigma_h\dx= \int_{\partial T}\sigma_h\cdot n_T \ds$ and we use the  flux definition \textcolor{black}{(\ref{eq: def_flux_IRT0_noncut})-(\ref{eq: def_flux_IRT0})}. As previously, we obtain \eqref{eq: conservation_property} on a non-cut cell $T$ by testing the mixed formulation \textcolor{black}{(\ref{eq: Mixed formulation})} with $(\chi_T,0)$ if $T\in \T^1$, and with $(0,\chi_T)$ if $T\in \T^2$. Finally, if $T\in \mathcal T_h^{\Gamma}$ we use the test-function $v_h=(\chi_T,\chi_T)$ in \textcolor{black}{problem (\ref{eq: Mixed formulation})}. The definition of \textcolor{black}{$v_h=(v_h^1,v_h^2)$ yields that on any cell $T'\in \mathcal T_h$ one has that $(\nabla v_h^i)_{|T'}=0$, hence} $a_i(u_h^i,v_h^i)=j_i(u_h^i,v_h^i)=0$ for $i=1,2$. \textcolor{black}{Moreover, one also has that
$$a_{\Gamma}(u_h,v_h)= \sum_{T'\in\T^{\Gamma}}\int_{\Gamma_{T'}}( \frac{\gamma k_{\Gamma}}{h_{T'}}[u_h]-\{K\nabla u_h\cdot n_{\Gamma}\})[v_h]\ds=0,$$
since for any cut cell $T'\in\T^{\Gamma}$, one has $(v_h^1)_{|T'}=(v_h^2)_{|T'}$ and therefore, $[v_h]_{|\Gamma_{T'}}=0$.}
Using next that for any $F\subset \partial T$, one can write that $\sigma_h\cdot n_T =\sigma_h\cdot n_F [\![v_h]\!]$, we get:
\begin{equation*}
			-\int_{T}\div(\sigma_h)\dx=-\sum_{F\in \partial T} \int_{F}\sigma_h\cdot n_F[\![v_h]\!]\ds=\textcolor{black}{-}d_h(u_h,v_h)+b_h(\theta_h,v_h),
\end{equation*} 
\textcolor{black}{where the last equality holds true thanks to \eqref{eq: def_flux_IRT0} and to the fact that $[\![u_h^i]\!]_{|F}=0$ for any $F\in \mathcal{F}_h^i$ ($1\le i\le 2$). Since $a_h(u_h,v_h)=0$ as previously shown, it follows that}
\begin{equation*}	-\int_{T}\div(\sigma_h)\dx=\textcolor{black}{\tilde{a}}_h(u_h,v_h)+b_h(\theta_h,v_h)=l_h(v_h)=\sum_{i=1}^{2}\int_{T^i}f^i\dx,
\end{equation*} 
which yields the desired relation: $-\displaystyle{\int_{T}}\div(\sigma_h)\dx=\displaystyle{\int_{T}}f\dx$.
\end{proof}

\section{Application to a posteriori error analysis}
Let  $\tau_h=K^{-1/2}(\sigma_h-K\nabla u_h)$. In addition to the standard a posteriori  error estimator $\eta_{T}=\|\tau_h\|_{T} $ for any $T\in \mathcal T_h$, we introduce two other local estimators:
\begin{equation*}
	\eta_F=\frac{\sqrt{h_F}}{\sqrt{k_{\Gamma}}}\|[\![\sigma_h\cdot n_F-\pi_F^0\sigma_h\cdot n_F]\!]\|_{F}\,\, \forall F\in \mathcal F_h^{\Gamma},\quad 
	\tilde{\eta}_{T}=\frac{\sqrt{k_{max}}}{h_T}\|[u_h]\|_{T}\,\, \forall T\in \T^\Gamma.
\end{equation*}
We define the corresponding global error estimators and the data approximation by:
\begin{equation*}
	\eta^2 =\sum_{T\in \mathcal{T}_h}\eta_{T}^2,\quad 
	\eta_{\Gamma}^2=\sum_{F\in \F^\Gamma}\eta_F^2+\sum_{T\in \T^\Gamma}\tilde{\eta}_{T}^2,\quad \epsilon(\Omega)^2= \sum_{T\in \mathcal{T}_h}\frac{h_T^2}{k_T}\|f-\pi_T^0f\|_{T}^2,
\end{equation*}
where $k_T=k_i$ if $T\in \T^{i}\backslash\T^\Gamma$ and $k_T=k_{\Gamma}$ if $T\in \T^\Gamma$. For simplicity of notation, for  any $v=(v^1,v^2)\in H^1(\Omega^1)\times H^1(\Omega^2)$ we denote $|v|_{1,K}^2=\displaystyle{\sum_{i=1}^2} k_i\|\nabla v^i\|_{\Omega^i}^2$.

\begin{theorem}[Reliability]
Let $u$  and $u_h$ the solutions of \textcolor{black}{\eqref{eq: continuous_problem_weak}} and \eqref{eq:Primal_conform_Formulation}, respectively. There exists a constant $C$ independent of the mesh, coefficients and  geometry s.t.  
	\begin{equation}
		|u-u_h|_{1,K}\leq \eta +C\left(\eta_{\Gamma}+\epsilon(\Omega)\right).
	\end{equation}
\end{theorem}

\begin{proof}
\textcolor{black}{Let $\sigma=K\nabla u$ and} let $\varphi \in H^1_0(\Omega)$ the unique solution of the weak problem: 
	\begin{equation}\label{eq: aposteriori_phi}
		\int_{\Omega}K\nabla\varphi\cdot \nabla v \dx = \textcolor{black}{\sum_{i=1}^2} \int_{\Omega^{\textcolor{black}{i}}} K\nabla u_h^i \cdot \nabla v \dx, \quad \forall v \in H^1_0(\Omega).
	\end{equation}
By the triangle inequality, we have
	$|u-u_h|_{1,K} \leq  |u-\varphi|_{1,K}+ |\varphi-u_h|_{1,K}$
with 
\begin{equation*}
\begin{split}
|u-\varphi|_{1,K}^2=&\textcolor{black}{\sum_{i=1}^2\int_{\Omega^i}\nabla(u-\varphi)\cdot (K\nabla u-\sigma_h+K^{1/2}\tau_h+K\nabla u_h^i-K\nabla \varphi)\dx}\\
=&\int_{\Omega}\nabla(u-\varphi)\cdot (\sigma-\sigma_h)\dx+\int_{\Omega} K^{1/2}\nabla(u-\varphi)\cdot \tau_h\dx\\
&+\sum_{i=1}^2\int_{\Omega^i}\nabla(u-\varphi)\cdot K\nabla(u_h^i-\varphi)\dx.
\end{split}
\end{equation*}
	The last term vanishes thanks to \eqref{eq: aposteriori_phi}, so \textcolor{black}{the Cauchy-Schwartz inequality gives} 
	\begin{equation}\label{eq: aposteriori_step2}
		|u-\varphi|_{1,K}^2\leq \left \vert \int_{\Omega}\nabla(u-\varphi)\cdot (\sigma-\sigma_h)\dx \right \vert +\eta  |u-\varphi|_{1,K}.
	\end{equation}
	\textcolor{black}{Since $[\![\sigma_h\cdot n_F]\!]_{|F}=0$ on any non-cut edge $F$,} an integration by parts next yields:
	\begin{equation}\label{eq: aposteriori_step3}
		\int_{\Omega}\nabla(u-\varphi)\cdot (\sigma-\sigma_h)\dx =\int_{\Omega}(f-\pi_T^0f)(u-\varphi)\dx -\sum_{F\in\F^\Gamma} \int_{F}[\![\sigma_h\cdot n_F]\!](u-\varphi)\ds.
	\end{equation}
	Using the property of the orthogonal projection, the Cauchy-Schwartz inequality and the fact that $k_{\Gamma}\leq k_i$ for $i\in \{1,\, 2\}$, we obtain in a standard way that:
	\begin{equation}\label{eq: aposteriori_bound_of_f-fh}
		\left\vert  \int_{\Omega}(f-\pi_T^0f)(u-\varphi)\dx\right\vert \leq C |u-\varphi|_{1,K}\, \epsilon(\Omega).
	\end{equation}
	The weak continuity \eqref{eq:dof_IRT} of $\sigma_h\cdot n_F$ across the cut \textcolor{black}{edges} yields, for any $F\in \F^\Gamma$, 
	\begin{equation*}
		\int_{F}[\![\sigma_h\cdot n_F]\!](u-\varphi)\ds = \int_{F}([\![\sigma_h\cdot n_F]\!]-\pi_F^0[\![\sigma_h\cdot n_F]\!])(u-\varphi-\pi_F^0(u-\varphi))\ds.
	\end{equation*}
	Using again $ k_\Gamma\leq k_i$ for $i\in\{1,\, 2\}$ and classical error bounds,
	we get 
	\begin{equation}\label{eq: aposteriori_bound_of_sigma_h}
		\bigg\vert \sum_{F\in\F^\Gamma} \int_{F}[\![\sigma_h\cdot n_F]\!](u-\varphi)\ds\bigg\vert \leq C |u-\varphi|_{1,K}\bigg(\sum_{F\in \F^\Gamma}\eta_F^2\bigg)^{1/2}.
	\end{equation}
	By gathering together \eqref{eq: aposteriori_step2}, \textcolor{black}{\eqref{eq: aposteriori_step3},} \eqref{eq: aposteriori_bound_of_f-fh} and \eqref{eq: aposteriori_bound_of_sigma_h}, we obtain that:
	\begin{equation}\label{eq: aposteriori_bound_1}
		|u-\varphi|_{1,K}\leq \eta + C ( \eta_{\Gamma}+ \epsilon(\Omega)).
	\end{equation}
	
	As regards the remaining term  $|\varphi-u_h|_{1,K}$, we have by definition of $\varphi$ that 
	\begin{equation*}
		|\varphi-u_h|_{1,K}= \inf_{v\in H_0^1(\Omega)}|v-u_h|_{1,K}\leq |U_h-u_h|_{1,K},
	\end{equation*}
	where $U_h$ is any approximation of $u_h$ in $H_0^1(\Omega)$. We choose $U_h$ piecewise linear and continuous, defined by $U_h(N)=u_h^i(N)$ if $N\in \N^i$. Then we have:
	\begin{equation*}
		\begin{split}
			|U_h-u_h|_{1,K}^2 \leq & \sum_{T\in \T^\Gamma}\sum_{i=1}^2 k_i|U_h-u_h^i|^2_{1,T} \leq  C \sum_{T\in \T^\Gamma} \sum_{i=1}^2 \sum_{N\in \mathcal N_T}k_i (U_h-u_h^i)^2(N)\\
			\le & C  \sum_{T\in \T^\Gamma} k_{max} \sum_{N\in \mathcal N_T}(u_h^1-u_h^2)^2(N)\le C\sum_{T\in \T^\Gamma}\frac{k_{max}}{h_T^2}\|[u_h]\|_T^2, 
		\end{split}
	\end{equation*}
	hence $ |\varphi-u_h|_{1,K}\leq  C \eta_{\Gamma}$, which together with \eqref{eq: aposteriori_bound_1} ends the proof.
\end{proof}
\begin{acknowledgement}
This project has received funding from the European Union’s Horizon H2020 Research and Innovation under the Marie Curie Grant Agreement N° 945416.
\end{acknowledgement}

\end{document}